\documentclass[11pt,letterpaper]{article}
\usepackage[margin=1in]{geometry}
\usepackage[T1]{fontenc}
\usepackage{lmodern}

\RequirePackage{amsthm,amsmath,amsfonts,amssymb,mathtools}
\RequirePackage[numbers,sort&compress]{natbib}
\RequirePackage{booktabs,tabularx,longtable,microtype,xcolor,graphicx,placeins,needspace}
\RequirePackage[colorlinks,linkcolor=blue,citecolor=blue,urlcolor=blue]{hyperref}
\RequirePackage[nameinlink,capitalise,noabbrev]{cleveref}
\providecommand{\doi}[1]{doi: \href{https://doi.org/#1}{\begingroup\urlstyle{rm}\nolinkurl{#1}\endgroup}}
\hypersetup{
  pdftitle={Shifted Anticoncentration for Real Gram Hafnians and Symmetric Gaussian Hafnians},
  pdfauthor={Hongru Zhao},
  pdfsubject={Probability theory and random matrices},
  pdfkeywords={hafnian, Gaussian Gram matrix, small ball probability, anticoncentration}
}

\allowdisplaybreaks
\numberwithin{equation}{section}

\newtheorem{theorem}{Theorem}[section]
\newtheorem{proposition}[theorem]{Proposition}
\newtheorem{lemma}[theorem]{Lemma}
\newtheorem{corollary}[theorem]{Corollary}
\theoremstyle{definition}

\theoremstyle{definition}

\newtheorem*{unnumberedremark}{Remark}

\newcommand{\R}{\mathbb R}
\newcommand{\E}{\mathbb E}
\newcommand{\Pp}{\mathbb P}
\newcommand{\M}{\mathcal M}
\newcommand{\haf}{\operatorname{haf}}
\newcommand{\pf}{\operatorname{pf}}
\newcommand{\Gram}{\operatorname{Gram}}

\newcommand{\dd}{\,\mathrm d}
\newcommand{\T}{\mathsf T}

\newcommand{\norm}[1]{\lVert #1\rVert}
\newcommand{\abs}[1]{\lvert #1\rvert}

\begin{document}
\title{Shifted Anticoncentration for Real Gram Hafnians\\
and Symmetric Gaussian Hafnians}
\author{Hongru Zhao\\
\small School of Statistics, University of Minnesota\\
\small Minneapolis, Minnesota, USA\\
\small\texttt{zhao1118@umn.edu}}
\date{}
\maketitle

\begin{abstract}
We prove a uniform shifted anticoncentration theorem for the hafnian of a real
Gaussian Gram matrix under an explicit condition on the row dimension. After
normalization by its root mean square, the law has a bounded continuous density,
maximal at zero, and every interval has probability bounded by an explicit
coefficient times its radius. Under suitable growth conditions on the row dimension,
this coefficient grows at most polynomially in the hafnian order, meaning half the
dimension of the Gram matrix. We also compute the exact second moment. The
proof exploits the perfect matching structure, combining conditional Gaussian
representations, row suspension, and bilinear interpolation to control an
inverse moment of the conditional variance. At fixed hafnian order, a rescaled
limit as the row dimension grows yields corresponding bounds for the hafnian of
a real symmetric Gaussian matrix with independent entries above the diagonal.
\end{abstract}



\section{Introduction}

Anticoncentration bounds the probability that a random variable lies in a
short interval.  An important motivation comes from arguments for quantum
advantage in sampling.  In boson sampling, collision free output
probabilities are squared moduli of matrix permanents
\cite{AaronsonArkhipov2013}; in Gaussian boson sampling with squeezed
vacuum inputs, they are proportional to squared moduli of hafnians
\cite{HamiltonEtAl2017,KruseEtAl2019}.

For approximate sampling, a small total-variation error controls output
probabilities additively, whereas the relevant average-case hardness
conjectures concern relative approximation.  Anticoncentration controls
how often a probability is too small for an additive estimate to give
useful relative accuracy.  Combined with Stockmeyer's approximate counting
method, this conversion relates a hypothetical efficient classical sampler
to relative estimates of output probabilities on a controlled fraction of
instances \cite{Stockmeyer1985,AaronsonArkhipov2013,HangleiterEisert2023}.
Anticoncentration is one ingredient in these conditional hardness
arguments, together with average-case hardness and a justified passage
between optical and Gaussian matrix ensembles.

Gaussian determinants and Pfaffians illustrate why analogous probabilistic
estimates can be more tractable.  The orthogonal reductions in
Appendix~\ref{app:benchmarks}, Sections~\ref{ex:real-ginibre-determinant}
and~\ref{app:df-axiom}, give the independent chi product laws
\eqref{eq:determinant-product-law} and \eqref{eq:df-pfaffian-law-axiom}.
Their squares are products of independent gamma variables.  The signed
laws give Gaussian scale mixtures, reducing density and short interval
bounds to explicit inverse moments.  These reductions preserve determinant
magnitude under orthogonal row transformations and Pfaffian magnitude under
orthogonal congruence.  For permanents and hafnians, the corresponding
transformations of their matrix arguments do not preserve magnitude in
general.  The same reductions therefore do not furnish independent product
laws in general, and their permutation or matching summands share entries.

General anticoncentration bounds for Gaussian polynomials deteriorate with
the degree \cite{CarberyWright2001,MekaNguyenVu2016}, so bounds linear in
the interval radius require more specific structure.  Koehler and Leung's
independent-entry permanent method combines a bilinear Gaussian kernel,
averaged interpolation, and Fourier comparison
\cite[Lemmas~3.1--3.4 and Proposition~3.5]{KoehlerLeung2026}.
For Gaussian boson sampling, recent second-moment analyses describe
squeezing-dependent behavior \cite{EhrenbergEtAlPRL2025,EhrenbergEtAlPRA2025},
and auxiliary-field formulas give related moment representations
\cite{ShouEhrenbergEtAl2026}.  In the real Gram problem, the matching terms
overlap and the Gram entries reuse the same Gaussian columns.  Independence
of the underlying Gaussian entries therefore does not separate the terms,
and the permanent argument does not itself give the needed Gram cofactor
comparison.

We establish a bound linear in the interval radius, uniformly over its
center, for hafnians of real Gaussian Gram matrices.

Last column expansion makes the proof's main obstacle precise.  Conditional
on the preceding columns, the hafnian is centered Gaussian with variance
equal to the squared norm of the column matrix multiplied by its own
cofactor vector.  A bounded density therefore requires an inverse half
moment of this dependent random scale.  A worst-direction singular-value
estimate discards the matching structure.  Instead, we expose two columns,
classify their matching partners, and apply bilinear Gaussian
interpolation.  A finite coordinate compression ends at a singleton,
where a forced matching edge removes one ambient row and lowers the
hafnian order by one.  This matching decomposition and deletion endpoint
supply the Gram cofactor comparison.

As an optical application, we treat a prescribed, or independently
selected, output pattern in an equally squeezed real orthogonal sector.
The Haar block approximation has an explicit additive total-variation
error.  This gives a probability bound for that sector; separate
complexity assumptions are required to turn such bounds into an argument
for quantum advantage.

The main results and a compact classical comparison appear next.  We then
prove the cofactor, compression, Fourier, and density steps in order,
followed by the normalization, symmetric limit, and coefficient estimates.
The appendices contain the Gaussian kernel calculation, classical
benchmarks, and the full optical derivation.

Theorems~\ref{thm:main} and~\ref{thm:symmetric-limit} are formalized in
Lean~4 without additional mathematical axioms; their proofs are available
on GitHub \cite{ZhaoRealGramHafniansGitHub2026}.  The broader Zenodo
development \cite{ZhaoLeanVerification2026} supplies the supporting proofs
and a correspondence for every numbered theorem and equation, distinguishing
proved statements, definitions, hypotheses, and cited inputs.  It also
proves the exact Barnes identity and asymptotic in
\eqref{eq:bnR-barnes-asymptotic}.  The Pfaffian benchmark is conditional on
its cited product law; the permanent bound and Haar block approximation
remain outside the formal coverage.

\section{Main results}
\label{sec:results}

Throughout, vectors in $\R^d$ or $\mathbb C^d$ are columns, and
${}^{\T}$ denotes transpose without complex conjugation.  Coordinate
lists such as $(z_i)_{i=1}^d$ denote column vectors;
$[v_1\ \cdots\ v_m]$ denotes the matrix with columns $v_1,\ldots,v_m$.
For real vectors $u,v\in\R^d$, we write
$\langle u,v\rangle=u\cdot v=u^{\T}v=\sum_{a=1}^d u_av_a$
and use the Euclidean norm.

For $n,k\ge1$, let $X\in\R^{k\times2n}$ have independent standard normal
entries, and let $x_1,\ldots,x_{2n}\in\R^k$ be its columns.
Let $\M_{2n}$ denote the perfect matchings of $\{1,\ldots,2n\}$.
Since diagonal entries do not enter the hafnian, throughout the paper
\begin{equation}
 H_{k,n}=
 \sum_{\mathfrak m\in\M_{2n}}
 \prod_{\{i,j\}\in\mathfrak m}\langle x_i,x_j\rangle.
 \label{eq:H-definition}
\end{equation}
For $d>1$, let $\chi_d^2$ denote a random variable having the chi-square
distribution with $d$ degrees of freedom, and define
\begin{equation}
 \gamma_d
 :=\E[(\chi_d^2)^{-1/2}]
 =2^{-1/2}\frac{\Gamma((d-1)/2)}{\Gamma(d/2)}.
 \label{eq:gamma-definition}
\end{equation}
Write $\sigma_{k,n}:=(\E H_{k,n}^2)^{1/2}$ for the root mean square.
For $n\ge1$ and $k\ge n+1$, define
\begin{equation}
 B^{\R}_{k,n}
 :=\sqrt{\frac2\pi}\,\sigma_{k,n}
 \left(\prod_{j=0}^{n-1}\gamma_{k-j}\right)
 \left(\prod_{r=2}^{n}\gamma_{2r-1}\right).
 \label{eq:exact-B}
\end{equation}
Empty products are interpreted as one.  The gamma identity follows by
integrating the chi-square density.  Section~\ref{sec:moment} proves the
real normalization \eqref{eq:exact-rms} directly, using the same
auxiliary-field method as the complex proof in
\cite[Eq.~(3), Supplemental Eqs.~(S8), (S9), (S11)]{ZhaoComplexMoments2026},
with conjugation omitted for real columns.  The complex normalization is
also given in \cite[Theorem~1]{EhrenbergEtAlPRA2025}.

\begin{theorem}[Density and explicit shifted anticoncentration]
\label{thm:main}
Let $n\ge1$ and $k\ge n+1$, and let $H_{k,n}$,
$\sigma_{k,n}$, and $B^{\R}_{k,n}$ be as defined above.
Then $\sigma_{k,n}>0$ and
\begin{equation}
 \sigma_{k,n}^{2}=\E H_{k,n}^{2}
 =(2n-1)!!\prod_{q=0}^{n-1}(k+2q).
 \label{eq:exact-rms}
\end{equation}
The law of $H_{k,n}$ has a bounded continuous density $f_{k,n}$ satisfying
\begin{equation}
 \sup_{z\in\R} f_{k,n}(z)=f_{k,n}(0)
 \le \frac{B^{\R}_{k,n}}{2\sigma_{k,n}}.
 \label{eq:density-bound}
\end{equation}
Consequently, for every $z\in\R$ and $\varepsilon\ge0$,
\begin{equation}
 \Pp\{\abs{H_{k,n}-z}\le\varepsilon\sigma_{k,n}\}
 \le \min\{1,B^{\R}_{k,n}\varepsilon\}.
 \label{eq:exact-small-ball}
\end{equation}
If, in addition, $k\ge n+2$, then
\begin{equation}
 B^{\R}_{k,n}
 \le
 \frac2{\sqrt\pi}n^{3/8}
 \exp\!\left\{\frac{3n^{2}+n}{4(k-n-1)}\right\}.
 \label{eq:elementary-B}
\end{equation}
\end{theorem}

\begin{proof}
Sections~\ref{sec:cofactor}--\ref{sec:fourier} prove cofactor positivity and
the inverse moment estimate.  Section~\ref{sec:density-proof} then gives the
density and shifted bound, Section~\ref{sec:moment} computes the exact second
moment, and Section~\ref{sec:constant} proves the elementary coefficient
estimate.
\end{proof}

\begin{unnumberedremark}[Optimality of the radius dependence]
The Gaussian mixture formula gives $f_{k,n}(0)>0$.  Continuity at zero
therefore implies
\[
 \lim_{\varepsilon\downarrow0}
 \frac{\Pp\{|H_{k,n}|\le\varepsilon\sigma_{k,n}\}}{\varepsilon}
 =2\sigma_{k,n}f_{k,n}(0)>0.
\]
Thus a uniform bound with $o(\varepsilon)$ radius dependence is impossible.
\end{unnumberedremark}

\Needspace{13\baselineskip}
\begin{corollary}[Polynomial small balls]
\label[corollary]{cor:polynomial}
Fix $A,D>0$ and let $(k_n)$ be a sequence of positive integers such that,
for all sufficiently large $n$,
\begin{equation}
 \frac{n^2}{k_n}\le D\log n.
 \label{eq:asymptotic-regime}
\end{equation}
For every fixed
\begin{equation}
 \alpha>A+\frac{3D}{4}+\frac38,
 \label{eq:alpha-choice}
\end{equation}
one has, for all sufficiently large $n$ and every $z\in\R$,
\begin{equation}
 \Pp\{\abs{H_{k_n,n}-z}\le n^{-\alpha}\sigma_{k_n,n}\}
 \le n^{-A}.
 \label{eq:polynomial-small-ball}
\end{equation}
\end{corollary}

\begin{proof}
The proof is given in Section~\ref{sec:constant}.  There the scale assumption is
inserted into \eqref{eq:elementary-B}; taking logarithms gives the stated
polynomial coefficient, and \eqref{eq:alpha-choice} yields the uniform small
ball conclusion \eqref{eq:polynomial-small-ball}.
\end{proof}

\subsection{The symmetric Gaussian model and classical comparisons}
Let $S_n$ be symmetric of size $2n$, with zero diagonal and independent
standard normal entries above the diagonal.  Define
\begin{equation}
 H_n:=\haf(S_n),\qquad
 \widehat H_n:=H_n/\sqrt{(2n-1)!!}.
 \label{eq:symmetric-hafnian-normalization}
\end{equation}
The diagonal convention is immaterial.  Put
\[
 b_n^{\R}:=\sqrt{\frac2\pi}\sqrt{(2n-1)!!}
                \prod_{r=2}^n\gamma_{2r-1}.
\]
\begin{theorem}[Real symmetric Gaussian anticoncentration]
\label{thm:symmetric-limit}
Fix $n\ge1$, let $S_n$ be the real symmetric Gaussian matrix defined above,
and define $H_n$ and $\widehat H_n$ by
\eqref{eq:symmetric-hafnian-normalization}.  Let
$\sigma_n^{\mathrm{sym}}$ be the root
mean square of $H_n$.  Then, as $k\to\infty$
with $n$ fixed,
\[
 k^{-n/2}H_{k,n}\ \Rightarrow\ H_n,
 \qquad
 k^{-n/2}\sigma_{k,n}\longrightarrow\sigma_n^{\mathrm{sym}},
\]
and
\[
 (\sigma_n^{\mathrm{sym}})^2=(2n-1)!!,
 \qquad
 B^{\R}_{k,n}\longrightarrow b_n^{\R}.
\]
Moreover, $\widehat H_n$ has a bounded continuous even density
$f_{\widehat H_n}$, this density is maximal at zero, and
\[
 \|f_{\widehat H_n}\|_\infty=f_{\widehat H_n}(0)
 \le \frac{b_n^{\R}}2\le\frac1{\sqrt\pi}n^{3/8}.
\]
For every $z\in\R$ and $\varepsilon\ge0$,
\[
 \Pp\{\abs{H_n-z}
       \le\varepsilon\sigma_n^{\mathrm{sym}}\}
 \le \min\{1,b_n^{\R}\varepsilon\}
 \le \min\!\left\{1,
       \frac2{\sqrt\pi}n^{3/8}\varepsilon\right\}.
\]
\end{theorem}

\begin{proof}
Section~\ref{app:symmetric-limit-weak} proves the weak-limit interval
bound.  Section~\ref{app:symmetric-limit-density} then obtains the inverse
moment and density by Fatou's lemma.  The finite $3/8$ estimate is proved
in Section~\ref{sec:constant}.
\end{proof}

For comparison, let $G_n^{\R}\in\R^{n\times n}$ have independent standard
normal entries, and let $A_n\in\R^{2n\times2n}$ be skew symmetric with
independent standard normal upper
entries.  Write $P_n=\operatorname{per}(G_n^{\R})$,
$D_n=\det(G_n^{\R})$, and $\mathrm{Pf}_n=\pf(A_n)$, where $\pf$ denotes
the Pfaffian, the signed sum over perfect matchings defined in
Appendix~\ref{app:four-cofactor}.  Hats denote division by the root mean
square.  Define
$\kappa_{n,1}:=\sqrt{n!}/(2^{n/2}\Gamma(n/2))$.
Table~\ref{tab:four-polynomials} compares density peaks, not interval
coefficients.  The permanent entry is
\cite[Theorem~1.1, Eqs.~(20), (26)]{KoehlerLeung2026}: their $\beta=1$
variable $W_n^{\R}$ is precisely $\widehat P_n$ with unit-variance entries.
The determinant and Pfaffian values are proved in
Appendix~\ref{app:benchmarks}, and the hafnian bound is
\cref{thm:symmetric-limit}.  Stirling's formula \cite[Eq.~(5.11.3)]{NISTDLMFStirling} gives the displayed
asymptotic for $\kappa_{n,1}$; \eqref{eq:bnR-barnes-asymptotic} gives the
one for $b_n^{\R}$.  Here $\mathcal G$ denotes the Barnes $G$ function,
recalled in Section~\ref{sec:exact-comparison-scales}.
An asymptotic beside a hafnian upper bound describes
the coefficient, not the unknown hafnian density peak.

\begin{table}[ht]
\caption{Peak densities of four root mean square normalized real Gaussian
matrix polynomials.  An equality is an exact value; an inequality is the
proved upper bound.}
\label{tab:four-polynomials}
\centering
\small
\setlength{\tabcolsep}{4pt}
\begin{tabularx}{\linewidth}{@{}>{\raggedright\arraybackslash}X l
  @{\hspace{16pt}}l@{}}
\toprule
Model & Normalized variable & Density maximum at zero \\
\midrule
Real Ginibre permanent
  & $\widehat P_n=P_n/\sqrt{n!}$
  & $\|f_{\widehat P_n}\|_\infty=f_{\widehat P_n}(0)
       \le\kappa_{n,1}\sim(8\pi)^{-1/4}n^{3/4}$ \\
Real Ginibre determinant
  & $\widehat D_n=D_n/\sqrt{n!}$
  & $\|f_{\widehat D_n}\|_\infty=f_{\widehat D_n}(0)
       =\kappa_{n,1}\sim(8\pi)^{-1/4}n^{3/4}$ \\
Real symmetric Gaussian hafnian
  & $\widehat H_n=H_n/\sqrt{(2n-1)!!}$
  & $\|f_{\widehat H_n}\|_\infty=f_{\widehat H_n}(0)
       \le b_n^{\R}/2\sim\mathcal G(1/2)\pi^{-1/4}n^{3/8}$ \\
Real skew symmetric Gaussian \mbox{Pfaffian}
  & $\widehat{\mathrm{Pf}}_n=\mathrm{Pf}_n/\sqrt{(2n-1)!!}$
  & $\|f_{\widehat{\mathrm{Pf}}_n}\|_\infty
       =f_{\widehat{\mathrm{Pf}}_n}(0)=b_n^{\R}/2
       \sim\mathcal G(1/2)\pi^{-1/4}n^{3/8}$ \\
\bottomrule
\end{tabularx}
\end{table}

The orderings are weak.  At $n=1$ each hafnian and Pfaffian is one
standard Gaussian entry.  At $n=2$, the matching expansions give
\[
 H_2=x_{12}x_{34}+x_{13}x_{24}+x_{14}x_{23},\qquad
 \mathrm{Pf}_2=x_{12}x_{34}-x_{13}x_{24}+x_{14}x_{23}.
\]
Changing $x_{13}$ to $-x_{13}$ preserves the independent Gaussian law,
so the raw and equally normalized laws agree.  The same sign-flip argument
applied to $ad+bc$ and $ad-bc$ proves the permanent--determinant agreement
at orders one and two.  Strict peak inequalities for higher orders are not
asserted.

\section{Gaussian mixtures and odd cofactors}
\label{sec:mixture}

\begin{lemma}[Gaussian-mixture principle]
\label[lemma]{lem:mixture}
Suppose $R>0$ almost surely and $X\mid R\sim N(0,R^2)$.  Then
\[
 f_X(x)=\frac1{\sqrt{2\pi}}\E\left[R^{-1}e^{-x^2/(2R^2)}\right]
\]
is a density, possibly infinite at zero.  If $m=\E R^{-1}<\infty$, it is
continuous, even, and nonincreasing in $|x|$, with maximum
$m/\sqrt{2\pi}$.  For $c>0$, its normalized density is $c f_X(cx)$, and for $\varepsilon\ge0$,
\[
 \sup_z\Pp\{|X-z|\le c\varepsilon\}
 \le\min\{1,\sqrt{2/\pi}\,cm\varepsilon\}.
\]
Conversely, if $\Pp\{|X|\le h\}\le2Lh$ for all $h>0$, then
$m\le\sqrt{2\pi}L$.  In particular, a bounded density exists if and only
if $m<\infty$.
\end{lemma}

\begin{proof}
This is the elementary scale-mixture mechanism of
\cite{AndrewsMallows1974}.  Let
$q_r(x)=(\sqrt{2\pi}r)^{-1}e^{-x^2/(2r^2)}$.
Conditioning and Tonelli give, for every Borel set $D$,
$\Pp\{X\in D\}=\int_D\E q_R(x)\,\dd x$.
Each kernel is even and decreases with $|x|$.  If $m<\infty$, domination
by $R^{-1}/\sqrt{2\pi}$ proves continuity and the maximum formula.
A change of variables gives the normalized density, and integration over
an interval of length $2c\varepsilon$ gives the probability bound.
For the converse, continuity of each conditional kernel and Fatou give
\begin{equation}
 \frac{\E R^{-1}}{\sqrt{2\pi}}
 \le\liminf_{h\downarrow0}\E\left[\frac1{2h}\int_{-h}^h q_R(x)\,\dd x\right]
 =\liminf_{h\downarrow0}\frac{\Pp\{|X|\le h\}}{2h}\le L.
 \label{eq:mixture-fatou}
\end{equation}
The argument uses only nonnegative integrals until finiteness is established.
\end{proof}

\subsection{Odd cofactors and row suspension}
\label{sec:cofactor}

For column vectors $v_1,\ldots,v_m$ of the same dimension, write
$\Gram(v_1,\ldots,v_m)=(\langle v_i,v_j\rangle)_{i,j=1}^{m}$.
For $r\ge1$ put $m_r=2r-1$.  Let
$A=[a_1\ \cdots\ a_{m_r}]\in\R^{k\times m_r}$ and set
$S=A^{\T}A$.  Define the entries of the odd hafnian cofactor column
$C_{k,r}(A)=(C_{k,r,i}(A))_{i=1}^{m_r}\in\R^{m_r}$ by
\begin{equation}
 C_{k,r,i}(A)=\haf(S_{-i}),
 \qquad 1\le i\le m_r,
 \label{eq:cofactor-vector}
\end{equation}
where $S_{-i}$ deletes row and column $i$.  The column
$q_{k,r}(A)\in\R^k$ and its scalar squared norm are
\begin{equation}
 q_{k,r}(A)=\sum_{i=1}^{m_r}C_{k,r,i}(A)a_i=AC_{k,r}(A),
 \qquad
 V_{k,r}(A)=\norm{q_{k,r}(A)}^2.
 \label{eq:qV-definition}
\end{equation}

If $x$ is an additional column, hafnian expansion along $x$ gives
\begin{equation}
 \haf\Gram(a_1,\ldots,a_{m_r},x)
 =\sum_{i=1}^{m_r}\langle a_i,x\rangle C_{k,r,i}(A)
 =\langle q_{k,r}(A),x\rangle.
 \label{eq:last-column-expansion}
\end{equation}
In particular, conditionally on $A$, the left side of
\eqref{eq:last-column-expansion} is $N(0,V_{k,r}(A))$ when $x$ is an
independent standard Gaussian column.

The same expansion has a useful deterministic version.  Fix the first
standard basis column $e_1\in\R^k$ and define the augmented hafnian
\begin{equation}
 P_{k,r}(e_1;A):=\haf\Gram(a_1,\ldots,a_{m_r},e_1).
 \label{eq:augmented-definition}
\end{equation}
Expanding along $e_1$ yields the row suspension identity
\begin{equation}
 P_{k,r}(e_1;A)=\langle e_1,q_{k,r}(A)\rangle.
 \label{eq:row-suspension-identity}
\end{equation}
No generality is lost: the Gaussian column law is orthogonally invariant,
and \cref{lem:radiality} makes this reduction explicit.

\begin{lemma}[Radiality]
\label[lemma]{lem:radiality}
If the columns of $A$ are independent standard Gaussians, then the law of
$q_{k,r}(A)$ is invariant under the orthogonal group $O(k)$.
\end{lemma}

\begin{proof}
For $U\in O(k)$, one has $(UA)^{\T}(UA)=A^{\T}A$.  Hence the cofactor
vector is unchanged and $q_{k,r}(UA)=Uq_{k,r}(A)$.  Since $UA$ and $A$
have the same law, the claim follows.
\end{proof}

The following basic fact handles radius zero and justifies the density
mixture without a hidden singular component.

\begin{lemma}[Almost sure positivity]
\label[lemma]{lem:positivity}
If $r,k\ge1$, then
\begin{equation}
 \Pp\{V_{k,r}=0\}=0.
 \label{eq:V-positive}
\end{equation}
Moreover $\Pp\{H_{k,r}=0\}=0$ for the $2r$-column Gram hafnian.
\end{lemma}

\begin{proof}
The first coordinate of $q_{k,r}$ is a polynomial in the entries of $A$.
At $a_1=\cdots=a_{2r-1}=e_1$, every cofactor counts the matchings of
$2r-2$ vertices, so
\[
 q_{k,r}(A)=(2r-1)!!\,e_1\ne0.
\]
Here $(-1)!!=1$ handles $r=1$.  The polynomial is therefore nonzero.
Its zero set is Lebesgue null \cite{Mityagin2020ZeroSet}, hence Gaussian
null.  Thus $q_{k,r}\ne0$ almost surely.  Conditional on $A$,
\eqref{eq:last-column-expansion} is nondegenerate Gaussian, proving the
hafnian nonvanishing assertion as well.
\end{proof}

\section{An exact bilinear Gaussian interpolation}
\label{sec:kernel}

The analytic engine of the compression is elementary but exact.  The kernel below is the specialization of
\cite[Lemmas~3.1--3.2]{KoehlerLeung2026} with
$T_{\rm KL}=tT$, $L=I$, $a=tV$, and $b=tU$.
We include its proof to fix the variables used by the matching decomposition.

\begin{lemma}[Bilinear Gaussian kernel]
\label[lemma]{lem:kernel}
Let $X,Y$ be independent standard Gaussian column vectors in $\R^d$, let
$T\in\R^{d\times d}$, $U,V\in\R^d$, and $t\in\R$.  Put
\begin{equation}
 F_{t,T}(U,V)
 :=\E\exp\{it(X^{\T}TY+X\cdot U+Y\cdot V)\}.
 \label{eq:F-definition}
\end{equation}
Set $D=I_d+t^2TT^{\T}$ and $E=I_d+t^2T^{\T}T$.  Then
\begin{align}
 F_{t,T}(U,0)
 &=\det(D)^{-1/2}
 \exp\!\left\{-\frac{t^2}{2}U^{\T}D^{-1}U\right\},
 \label{eq:F-left}\\
 F_{t,T}(0,V)
 &=\det(E)^{-1/2}
 \exp\!\left\{-\frac{t^2}{2}V^{\T}E^{-1}V\right\}.
 \label{eq:F-right}
\end{align}
Both quantities are strictly positive.  For every $0\le\theta\le1$,
\begin{equation}
 \abs{F_{t,T}(\sqrt\theta\,U,\sqrt{1-\theta}\,V)}
 =F_{t,T}(U,0)^\theta F_{t,T}(0,V)^{1-\theta}.
 \label{eq:geometric-interpolation}
\end{equation}
\end{lemma}

\begin{proof}Appendix~\ref{app:kernel} gives the Gaussian integral calculation.\end{proof}

\section{Two coordinate compression}
\label{sec:compression}

Use the fixed vector $e_1$ from \eqref{eq:augmented-definition} and
decompose the Gaussian columns as
\begin{equation}
 a_i=\begin{pmatrix}y_i\\ b_i\end{pmatrix}\in\R^k,
 \qquad y_i\in\R,
 \qquad b_i\in\R^{k-1}.
 \label{eq:row-splitting}
\end{equation}
Collect the first coordinates in the column
$y=(y_1,\ldots,y_{m_r})^{\T}\in\R^{m_r}$, and let
$B=[b_1\ \cdots\ b_{m_r}]\in\R^{(k-1)\times m_r}$ be the matrix of lower
columns.  Thus
\[
 A=\begin{pmatrix}y^{\T}\\ B\end{pmatrix},
\]
and the first row of $A$ is $y^{\T}$.  For deterministic $y$ and independent
standard Gaussian columns of $B$, write
\begin{equation}
 P_r(y,B):=P_{k,r}(e_1;A),
 \qquad
 \Psi_{r,t}(y):=\E_B e^{itP_r(y,B)}.
 \label{eq:Psi-definition}
\end{equation}
The dependence on $k$ is suppressed in this section.

\subsection{The local matching decomposition}

Fix distinct $p,q\in\{1,\ldots,m_r\}$.  Put
$x=y_p$, $y'=y_q$, $X=b_p$, and $Y=b_q$, so $x,y'\in\R$ and
$X,Y\in\R^{k-1}$ are columns.
For the matching discussion, relabel the appended deterministic column
$a_0=e_1$ as index $0$.  Thus the augmented index set is
$\{0,1,\ldots,m_r\}$, with $y_0=1$ and $b_0=0\in\R^{k-1}$.
For a subset obtained by
deleting indices, let $h$ denote the hafnian of the remaining augmented
Gram matrix; in particular, write $h_{pq}$ after deleting $p,q$ and
$h_{pq\alpha\beta}$ after also deleting $\alpha,\beta$.

With $J=\{0,\ldots,m_r\}\setminus\{p,q\}$, define the background matrix
$T\in\R^{(k-1)\times(k-1)}$, columns $u,v\in\R^{k-1}$, and scalar
$a\in\R$ by
\begin{align*}
 T&=h_{pq}I_{k-1}+
 \sum_{\substack{\alpha,\beta\in J\\\alpha\ne\beta}}
 h_{pq\alpha\beta}\,b_\alpha b_\beta^{\T},\\
 u&=\sum_{\substack{\alpha,\beta\in J\\\alpha\ne\beta}}
 h_{pq\alpha\beta}\,y_\beta b_\alpha,
 \qquad
 v=\sum_{\substack{\alpha,\beta\in J\\\alpha\ne\beta}}
 h_{pq\alpha\beta}\,y_\alpha b_\beta,\\
 a&=h_{pq}+
 \sum_{\substack{\alpha,\beta\in J\\\alpha\ne\beta}}
 h_{pq\alpha\beta}\,y_\alpha y_\beta.
\end{align*}

\begin{lemma}[Exposed pair decomposition]
\label[lemma]{lem:local-decomposition}
Conditionally on every variable except $X,Y$, the augmented hafnian has the
form
\begin{equation}
 P_r(y,B)=X^{\T}TY+y'X\cdot u+xY\cdot v+a x y',
 \label{eq:local-phase}
\end{equation}
where $T,u,v,a$ depend only on the fixed background and not on
$x,y',X,Y$.
\end{lemma}

\begin{proof}
Partition the perfect matchings according to the partners of the exposed
vertices $p$ and $q$.  If $p$ and $q$ are paired together, their edge has
weight $\langle a_p,a_q\rangle=xy'+X\cdot Y$.  Removing this edge leaves
an arbitrary perfect matching of the remaining vertices, whose total
weight is $h_{pq}$.  Thus this class contributes
\begin{equation*}
 (xy'+X\cdot Y)h_{pq}.
\end{equation*}
Otherwise, $p$ is paired with a unique $\alpha\in J$, and $q$ with a
unique $\beta\in J\setminus\{\alpha\}$.  The pair $(\alpha,\beta)$ is
ordered: $\alpha$ is the partner of $p$, whereas $\beta$ is the partner
of $q$.  Removing these two edges leaves an arbitrary perfect matching
on the vertices outside $\{p,q,\alpha,\beta\}$.  Conversely, each such
remaining matching extends uniquely by adjoining the two specified
edges.  Hence the contribution for this ordered pair is
\begin{align*}
 &(xy_\alpha+X\cdot b_\alpha)
   (y'y_\beta+Y\cdot b_\beta)h_{pq\alpha\beta}\\
 &\quad=h_{pq\alpha\beta}\bigl[
 xy'y_\alpha y_\beta+x y_\alpha(Y\cdot b_\beta)
 +y'y_\beta(X\cdot b_\alpha)
 +(X\cdot b_\alpha)(Y\cdot b_\beta)\bigr].
\end{align*}
Every matching belongs to exactly one of these classes, so none is omitted
or counted twice.  In particular, interchanging $\alpha$ and $\beta$
changes the two edges and specifies a different class.

Using $(X\cdot b_\alpha)(Y\cdot b_\beta)
=X^{\T}b_\alpha b_\beta^{\T}Y$ and collecting the four types of terms gives
\begin{align*}
 P_r(y,B)
 ={}&X^{\T}\left(h_{pq}I_{k-1}+
 \sum_{\substack{\alpha,\beta\in J\\\alpha\ne\beta}}
 h_{pq\alpha\beta}b_\alpha b_\beta^{\T}\right)Y
 +y'X\cdot\left(
 \sum_{\substack{\alpha,\beta\in J\\\alpha\ne\beta}}
 h_{pq\alpha\beta}y_\beta b_\alpha\right)\\
 &+xY\cdot\left(
 \sum_{\substack{\alpha,\beta\in J\\\alpha\ne\beta}}
 h_{pq\alpha\beta}y_\alpha b_\beta\right)
 +xy'\left(h_{pq}+
 \sum_{\substack{\alpha,\beta\in J\\\alpha\ne\beta}}
 h_{pq\alpha\beta}y_\alpha y_\beta\right).
\end{align*}
The four parenthesized expressions are precisely $T,u,v,a$, respectively,
which proves \eqref{eq:local-phase}.  Both $h_{pq}$ and
$h_{pq\alpha\beta}$ are computed after deleting the exposed vertices,
so they involve none of $x,y',X,Y$.  All remaining factors in the four
coefficients have indices in $J$; consequently the coefficients depend
only on the background.

The deterministic vertex $0$ is included in this same partition.
Since $y_0=1$ and $b_0=0$, an edge to $0$ has weight
$\langle a_i,a_0\rangle=y_i$: when $\alpha=0$ the first edge factor
above is $x$, and when $\beta=0$ the second is $y'$.
If neither partner is $0$, that vertex remains in the residual hafnian.
Thus edges incident to $0$ are accounted for with their correct weights,
without any additional case or term.
\end{proof}

\subsection{Local compression}

This is the coordinate compression step corresponding to the Fourier
interpolation of \cite[Corollary~3.3]{KoehlerLeung2026}; the endpoint map
below is specific to the exposed Gram matching decomposition.
For arbitrary $x,y'$, set $R=(x^2+(y')^2)^{1/2}$ and define endpoint
columns $y^{(p)},y^{(q)}\in\R^{m_r}$ by replacing the pair $(x,y')$ respectively with
$(\operatorname{sgn}(x)R,0)$ and
$(0,\operatorname{sgn}(y')R)$.  All other coordinates are unchanged.  If a coordinate is zero, its
sign-defined endpoint need not preserve the norm; that endpoint has
weight zero.  Norm preservation is used only when both coordinates
are nonzero.  Put
$\theta=(y')^2/(x^2+(y')^2)$, with the convention $\theta=0$ when
$x=y'=0$.

\begin{proposition}[Local compression]
\label[proposition]{prop:local-compression}
For $\Psi_{r,t}(y)=\E_B e^{itP_r(y,B)}$ defined in
\eqref{eq:Psi-definition}, with expectation over the Gaussian columns of
$B$ and $y$ fixed, one has
\begin{equation}
 \abs{\Psi_{r,t}(y)}
 \le
 \Psi_{r,t}(y^{(q)})^\theta
 \Psi_{r,t}(y^{(p)})^{1-\theta}.
 \label{eq:local-compression}
\end{equation}
The two endpoint values in \eqref{eq:local-compression} are real and
nonnegative.  The statement includes all zero coordinate cases; a factor
with exponent zero is interpreted as one.
\end{proposition}

\begin{proof}
Condition on the background in \cref{lem:local-decomposition}.  If $R=0$,
the original vector and both endpoints coincide, so the assertion is
immediate.  Assume $R>0$.  The definitions of $R$ and $\theta$ give
\begin{equation*}
 y'u=\sqrt\theta\,\operatorname{sgn}(y')Ru,
 \qquad
 xv=\sqrt{1-\theta}\,\operatorname{sgn}(x)Rv.
\end{equation*}
For this fixed background, set
\begin{equation*}
 J_q:=\E_{X,Y}e^{itP_r(y^{(q)},B)},
 \qquad
 J_p:=\E_{X,Y}e^{itP_r(y^{(p)},B)}.
\end{equation*}
At either endpoint one of $x,y'$ is zero, so the scalar term $axy'$ vanishes.
The endpoint formulas in \cref{lem:kernel}, with
$U_q=\operatorname{sgn}(y')Ru$ and
$V_p=\operatorname{sgn}(x)Rv$, therefore give $J_p,J_q>0$.  In the original
conditional expectation the factor $e^{itaxy'}$ is a pure phase, and
\cref{lem:kernel} gives the exact identity
\begin{align*}
 \abs{\E_{X,Y}e^{itP_r(y,B)}}
 &=J_q^\theta J_p^{1-\theta},
\end{align*}
Integrate the remaining background.  The triangle inequality followed by H\"older's
inequality gives
\begin{equation*}
 \abs{\Psi_{r,t}(y)}
 \le \E[J_q^\theta J_p^{1-\theta}]
 \le (\E J_q)^\theta(\E J_p)^{1-\theta}.
\end{equation*}
For $0<\theta<1$ the second inequality is H\"older with conjugate exponents
$1/\theta$ and $1/(1-\theta)$; when $\theta$ is zero or one it is simply the
corresponding endpoint identity.
The expectations are exactly the two endpoint characteristic functions and
remain nonnegative.
\end{proof}

\subsection{Global support reduction}

For two nonzero coordinates, each application of \cref{prop:local-compression} preserves
$\sum_i y_i^2$ and reduces the number of nonzero coordinates by at least
one.  This finite descent is the counterpart of the support compression in
\cite[Proposition~3.5]{KoehlerLeung2026}, with the Gram specific endpoints
supplied by the preceding proposition.

\begin{proposition}[Radial coordinate compression]
\label[proposition]{prop:global-compression}
Assume $r\ge2$.  For every $y\in\R^{m_r}$, $t\in\R$, and
$1\le j\le m_r$, writing $e_j\in\R^{m_r}$ for the $j$th standard basis
column,
\begin{equation}
 \abs{\Psi_{r,t}(y)}
 \le \Psi_{r,t}(\norm y e_j).
 \label{eq:global-compression}
\end{equation}
The right hand side is real and nonnegative.
\end{proposition}

\begin{proof}
Induct on the support size of $y$.  For support size at least two, apply
\cref{prop:local-compression} to two nonzero coordinates.  Each endpoint has
exactly one fewer nonzero coordinate.  Since both endpoint characteristic
functions are nonnegative, their weighted geometric mean is at most the
larger one; choose that endpoint and repeat.  The Euclidean norm is preserved
at every step, and the descent ends at a singleton.  Simultaneously
permuting the coordinates $y_i$ and the independent lower columns $b_i$
does not change the law, so column exchangeability moves the singleton to
the prescribed index $j$.  Its sign is irrelevant because the
explicit singleton formula \eqref{eq:singleton-laplace} below depends only
on its square.
\end{proof}

\subsection{The collision free endpoint}

At a singleton, the deterministic column has only one possible nonzero
partner.  This gives the deletion that drives the recursion.

\begin{lemma}[Singleton deletion]
\label[lemma]{lem:singleton}
Assume $r\ge2$.  For $s\in\R$ and any $1\le j\le m_r$, form
$B_{-j}\in\R^{(k-1)\times(2r-2)}$ by deleting the $j$th column $b_j$
from $B=[b_1\ \cdots\ b_{m_r}]$, keeping the other columns in order, and set
$H^{(j)}_{k-1,r-1}(B):=\haf(B_{-j}^{\T}B_{-j})$.  Thus $B_{-j}$ is obtained from $A$ by deleting its first row and its $j$th
column; $B_{-j}^{\T}B_{-j}$ is obtained from $B^{\T}B$ by deleting both
row $j$ and column $j$.  Then
\begin{equation}
 P_r(se_j,B)=sH^{(j)}_{k-1,r-1}(B),
 \label{eq:singleton-deletion}
\end{equation}
and $H^{(j)}_{k-1,r-1}(B)$ has the same distribution as
$H_{k-1,r-1}$.  Consequently,
\begin{equation}
 \Psi_{r,t}(se_j)
 =\E\exp\!\left\{-\frac{t^2s^2}{2}V_{k-1,r-1}\right\}.
 \label{eq:singleton-laplace}
\end{equation}
\end{lemma}

\begin{proof}
When $y=se_j$, the first row of $A$ is $se_j^{\T}$, and the deterministic
column $e_1\in\R^k$ is
orthogonal to every random column except column $j$.  Hence every nonzero
matching term pairs $e_1$ with that column and contributes the factor $s$.
After deleting this pair, all remaining first coordinates vanish and their
Gram products are those of the lower $(k-1)$-dimensional columns.  This
proves the literal deleted-column identity
\eqref{eq:singleton-deletion}.  Because deleting a fixed member of an
independent identically distributed family leaves $2r-2$ independent standard
Gaussian columns, $H^{(j)}_{k-1,r-1}(B)$ has the law of $H_{k-1,r-1}$.
Expand that lower hafnian along its last column and condition on its
predecessors; the result is centered Gaussian with variance
$V_{k-1,r-1}$, proving
\eqref{eq:singleton-laplace}.
\end{proof}

\section{Fourier comparison and the inverse moment recursion}
\label{sec:fourier}

We use the standard zero dimensional convention $\R^0=\{0\}$.  Its Gram
matrix is the zero matrix, and $V_{0,\ell}=0$ for $\ell\ge1$.  With this
convention, the lower energy in \cref{prop:fourier-laplace} is defined also
when $k=1$.

Average over $y\sim N(0,I_{m_r})$, independently of the lower rows.  By the
row suspension identity,
\begin{equation*}
 \E_{y,B}e^{itP_r(y,B)}
 =\E e^{it\langle e_1,q_{k,r}\rangle}.
\end{equation*}
The triangle inequality followed by \eqref{eq:global-compression} gives
\begin{equation*}
 \abs{\E e^{it\langle e_1,q_{k,r}\rangle}}
 \le \E_y\abs{\Psi_{r,t}(y)}
 \le \E_y\Psi_{r,t}(\norm y e_j).
\end{equation*}
Since $G_{m_r}:=\norm y^2\sim\chi_{m_r}^2$, the singleton formula evaluates
the final expectation and yields
\begin{equation}
 \abs{\E e^{it\langle e_1,q_{k,r}\rangle}}
 \le
 \E\exp\!\left\{-\frac{t^2}{2}G_{m_r}V_{k-1,r-1}\right\}.
 \label{eq:scalar-fourier}
\end{equation}
Here and below the auxiliary chi-square variable is independent of the
lower level cofactor energy.

By \cref{lem:radiality}, the characteristic function of $q_{k,r}$ depends
only on the norm of its Fourier argument.  Invariance under the orthogonal
map $q\mapsto-q$ also makes this characteristic function real.  Its value at
$\xi$ therefore equals its value at $\norm\xi e_1$, and
\eqref{eq:scalar-fourier} with $t=\norm\xi$ gives
\begin{equation}
 \operatorname{Re}\E e^{i\xi\cdot q_{k,r}}
 \le
 \E\exp\!\left\{-\frac{\norm\xi^2}{2}
 G_{m_r}V_{k-1,r-1}\right\},
 \qquad \xi\in\R^k.
 \label{eq:vector-fourier}
\end{equation}

\begin{proposition}[Fourier to Laplace comparison]
\label[proposition]{prop:fourier-laplace}
Assume $r\ge2$ and $k\ge1$.  Let $G_k\sim\chi_k^2$ be independent of all
preceding variables.  For every $s\ge0$,
\begin{equation}
 \E e^{-sV_{k,r}/2}
 \le
 \E e^{-sG_kG_{m_r}V_{k-1,r-1}/2}.
 \label{eq:laplace-comparison}
\end{equation}
\end{proposition}

\begin{proof}
The case $s=0$ is immediate.  For $s>0$, let $Z\in\R^k$ be an independent $N(0,I_k)$ column,
put $\xi=\sqrt{s}\,Z$, and average \eqref{eq:vector-fourier} in $\xi$.
The average of the left-hand side of \eqref{eq:vector-fourier} equals
$\E e^{-s\norm{q_{k,r}}^2/2}=\E e^{-sV_{k,r}/2}$, which is the
left-hand side of \eqref{eq:laplace-comparison}.  To average the right-hand
side of \eqref{eq:vector-fourier} over $\xi$, we condition on
$U=G_{m_r}V_{k-1,r-1}$.  The resulting conditional average is
\begin{equation*}
 \E_\xi e^{-U\norm\xi^2/2}=(1+sU)^{-k/2}
 =\E_{G_k}e^{-sUG_k/2}.
\end{equation*}
This proves \eqref{eq:laplace-comparison}.
\end{proof}

We now use the completely monotone representation
\begin{equation}
 \frac1{\sqrt{2\pi}}
 \int_0^\infty s^{-1/2}e^{-sx/2}\,\dd s
 =\frac1{\sqrt{x}},
 \qquad x>0.
 \label{eq:inverse-half-representation}
\end{equation}
Indeed, the substitution $u=sx/2$ reduces the integral to
$\Gamma(1/2)=\sqrt\pi$.  This positive Mellin Laplace passage follows the
same Fourier principle as \cite[Lemma~3.4]{KoehlerLeung2026}, but the
random variables on the right arise from singleton row deletion.
Tonelli's theorem applies in the extended nonnegative reals, so the next
argument does not presuppose finiteness of the inverse moment.

\begin{proposition}[Row suspension recursion]
\label[proposition]{prop:row-recursion}
For $r\ge2$ and $k\ge r+1$,
\begin{equation}
 \E V_{k,r}^{-1/2}
 \le
 \gamma_k\gamma_{2r-1}\E V_{k-1,r-1}^{-1/2}.
 \label{eq:row-recursion}
\end{equation}
Consequently, for $n\ge1$ and $k\ge n+1$,
\begin{equation}
 \E V_{k,n}^{-1/2}
 \le
 \left(\prod_{j=0}^{n-1}\gamma_{k-j}\right)
 \left(\prod_{r=2}^{n}\gamma_{2r-1}\right).
 \label{eq:inverse-product}
\end{equation}
\end{proposition}

\begin{proof}
Apply \eqref{eq:inverse-half-representation} to $V_{k,r}$ and use
Tonelli's theorem to interchange expectation and integration.  Since
$m_r=2r-1$, integrating \eqref{eq:laplace-comparison} with the nonnegative
weight $s^{-1/2}/\sqrt{2\pi}$ gives the explicit chain
\begin{align*}
 \E V_{k,r}^{-1/2}
 &=\frac1{\sqrt{2\pi}}\E\!\left[
   \int_0^\infty s^{-1/2}e^{-sV_{k,r}/2}\,\dd s\right]\\
 &=\frac1{\sqrt{2\pi}}\int_0^\infty
   s^{-1/2}\E e^{-sV_{k,r}/2}\,\dd s\\
 &\le\frac1{\sqrt{2\pi}}\int_0^\infty
   s^{-1/2}\E e^{-sG_kG_{2r-1}V_{k-1,r-1}/2}\,\dd s\\
 &=\E(G_kG_{2r-1}V_{k-1,r-1})^{-1/2}.
\end{align*}
The last equality follows by Tonelli and
\eqref{eq:inverse-half-representation} applied to the product.
All integrands are nonnegative, so these steps hold in the extended
nonnegative reals without assuming finiteness.  Independence of the three
factors and \eqref{eq:gamma-definition} then give
\begin{align*}
 \E(G_kG_{2r-1}V_{k-1,r-1})^{-1/2}
 &=(\E G_k^{-1/2})(\E G_{2r-1}^{-1/2})
   \E V_{k-1,r-1}^{-1/2}\\
 &=\gamma_k\gamma_{2r-1}\E V_{k-1,r-1}^{-1/2}.
\end{align*}
The extended-integral comparison already holds for $k\ge2$.
The stated range $k\ge r+1$ makes all iterated factors finite.
At level one,
$V_{d,1}\sim\chi_d^2$, so $\E V_{d,1}^{-1/2}=\gamma_d$.  Iterating from
$(k,n)$ through
$(k-1,n-1),\ldots,(k-n+1,1)$ proves
\eqref{eq:inverse-product}.
\end{proof}

\section{Density and normalization}
\label{sec:density-proof}
\subsection{The finite Gram density}
Conditioning on the first $2n-1$ columns in
\eqref{eq:last-column-expansion} gives
\begin{equation}
 H_{k,n}\mid A\sim N(0,V_{k,n}).
 \label{eq:conditional-normal}
\end{equation}
By \cref{lem:positivity,prop:row-recursion}, its positive scale has finite
inverse moment for $k\ge n+1$.  Thus \cref{lem:mixture} gives
\begin{equation}
 f_{k,n}(z)=\E\left[\frac1{\sqrt{2\pi V_{k,n}}}
                          e^{-z^2/(2V_{k,n})}\right].
 \label{eq:density-mixture}
\end{equation}
Its continuity, maximum at zero, and \eqref{eq:inverse-product} prove
\eqref{eq:density-bound}.  The interval conclusion of the same lemma,
with $c=\sigma_{k,n}$, proves \eqref{eq:exact-small-ball} with the coefficient
\eqref{eq:exact-B}.

\subsection{Exact second moment}
\label{sec:moment}

The complex counterpart of \eqref{eq:exact-rms} is given in
\cite[Section~III.1, Theorem~1]{EhrenbergEtAlPRA2025}.
Our direct real proof uses the auxiliary-field method of
\cite[Eq.~(3), Supplemental Eqs.~(S8), (S9), (S11)]{ZhaoComplexMoments2026},
which also treats complex matrices.  Its $M_1$ is the second absolute
moment of the hafnian.  The two-field representation and columnwise
contraction carry over to real columns, with no conjugation; both reduce
the normalization to the same rank one Gaussian integral evaluated below.

Let $g,h$ be independent standard Gaussian columns in $\R^k$, independent
of the columns $x_1,\ldots,x_{2n}$ of $X$, and put
\[
 \Phi_{n,k}(g,h;X)
 :=\prod_{i=1}^{2n}\langle g,x_i\rangle\langle h,x_i\rangle .
\]
For fixed $X$, Wick's matching rule \cite{Isserlis1918} applied separately
to $g$ and $h$ gives
\begin{equation}
 \E_{g,h}\Phi_{n,k}(g,h;X)=H_{k,n}(X)^2.
 \label{eq:aux-fields-square}
\end{equation}
Indeed, each expectation is a sum over perfect matchings of the $2n$
linear forms.  The $g$ field produces one copy of the hafnian of
$X^{\T}X$, and the independent $h$ field produces the other.
Related hafnian identities for products of real linear functionals are
developed in \cite{Frenkel2008}; see also
\cite[Eqs.~(2.2), (2.6)]{ShouEhrenbergEtAl2026} for the auxiliary-field
representation and its general-moment extension.

The integrand is a finite polynomial in jointly Gaussian coordinates and is
absolutely integrable.  Fubini therefore allows the field and column
integrals to be interchanged.  For fixed $g,h$, independence of the columns
and the elementary covariance identity
$\E_x[\langle g,x\rangle\langle h,x\rangle]=\langle g,h\rangle$ yield
\begin{equation}
 \E_X\Phi_{n,k}(g,h;X)=\langle g,h\rangle^{2n}.
 \label{eq:aux-columns}
\end{equation}
Combining the last two displays gives
\[
 \E H_{k,n}^2=\E_{g,h}\langle g,h\rangle^{2n}.
\]

It remains to evaluate this rank one Gaussian bilinear moment.  Conditional
on $g$, the inner product $\langle g,h\rangle$ is centered normal with
variance $\norm g^2$.  Hence its conditional even moment is
$(2n-1)!!\,\norm g^{2n}$.  Since $\norm g^2$ is chi-square with $k$
degrees of freedom, its $n$th moment is
$\prod_{q=0}^{n-1}(k+2q)$.  Thus
\begin{equation}
 \E_{g,h}\langle g,h\rangle^{2n}
 =(2n-1)!!\prod_{q=0}^{n-1}(k+2q).
 \label{eq:bilinear-moment}
\end{equation}
This is \eqref{eq:exact-rms}.  Together with \cref{sec:density-proof}, it
proves every clause of \cref{thm:main} except the elementary coefficient
estimate, which is proved in \cref{sec:constant}.
\section{The symmetric Gaussian limit}
\label{app:symmetric-limit}

\subsection{Fixed dimensional Gaussian limit and shifted bound}
\label{app:symmetric-limit-weak}

\begin{proof}[Proof of the limit and shifted bound in
\cref{thm:symmetric-limit}]
Fix $n$.  Index the coordinates of $\R^{\binom{2n}{2}}$ by the unordered
pairs $\{i,j\}$ with $1\le i<j\le2n$, ordered lexicographically.
For each row index $a$ of $X$, define the column
\[
 W_a:=\bigl(X_{ai}X_{aj}\bigr)_{1\le i<j\le2n}
 \in\R^{\binom{2n}{2}}.
\]
The vectors $W_a$ are independent and identically distributed, centered,
and have finite second moments.  Their covariance matrix is the identity.
Indeed, for two unordered pairs $\{i,j\}$ and $\{\ell,m\}$,
independence and centering of the Gaussian coordinates give
\[
 \E[X_{ai}X_{aj}X_{a\ell}X_{am}]
 =\begin{cases}
   1,&\{i,j\}=\{\ell,m\},\\
   0,&\{i,j\}\ne\{\ell,m\}.
  \end{cases}
\]
In the second case, at least one coordinate occurs to the first power.
The standard multivariate Lindeberg--Feller theorem for random vectors in a
fixed Euclidean space \cite[Proposition~2.27]{vanDerVaart1998} therefore gives
\[
 \left(k^{-1/2}(X^{\T}X)_{ij}\right)_{1\le i<j\le2n}
 \ \Rightarrow\
 \left(S_{n,ij}\right)_{1\le i<j\le2n}.
\]
Let $\Phi_n$ denote the matching polynomial on this strict upper triangular
edge vector.  The hafnian depends only on these coordinates; in particular,
the diagonal of $k^{-1/2}X^{\T}X$ is neither asserted to converge nor used.
The polynomial $\Phi_n$ is continuous and homogeneous of degree $n$, so the
continuous mapping theorem yields
\[
 k^{-n/2}H_{k,n}
 =\Phi_n\!\left(
   \left(k^{-1/2}(X^{\T}X)_{ij}\right)_{i<j}\right)
 \ \Rightarrow\ H_n.
\]

We next identify the normalization on both sides.  Expanding
$\E H_n^2$ as a double sum over perfect matchings, a pair of
distinct matchings has an edge belonging to exactly one of them.  The
corresponding centered independent Gaussian entry occurs once, so that cross
term has expectation zero.  A matching paired with itself contributes one.
There are $(2n-1)!!$ perfect matchings, and hence
\[
 (\sigma_n^{\mathrm{sym}})^2=(2n-1)!!.
\]
On the finite row side, \eqref{eq:exact-rms} gives
\[
 k^{-n}\sigma_{k,n}^2
 =(2n-1)!!\prod_{q=0}^{n-1}\left(1+\frac{2q}{k}\right)
 \longrightarrow(2n-1)!!.
\]
Taking nonnegative square roots proves
$k^{-n/2}\sigma_{k,n}\to\sigma_n^{\mathrm{sym}}$.

The gamma ratio asymptotic, equivalently the limiting case of Wendel's
inequality \cite{Wendel1948}, gives
\[
 \sqrt d\,\gamma_d\longrightarrow1
 \qquad(d\to\infty).
\]
Thus, for every fixed $0\le j<n$,
$\sqrt{k}\,\gamma_{k-j}\to1$.  Substituting this and the preceding root mean
square limit into \eqref{eq:exact-B} gives
\[
 B^{\R}_{k,n}
 =\sqrt{\frac2\pi}\,\bigl(k^{-n/2}\sigma_{k,n}\bigr)
   \prod_{j=0}^{n-1}\bigl(\sqrt{k}\,\gamma_{k-j}\bigr)
   \prod_{r=2}^{n}\gamma_{2r-1}
 \longrightarrow b_n^{\R}.
\]

It remains to transfer the interval bound.  Put
$Z_k:=k^{-n/2}H_{k,n}$, $Z:=H_n$, and
$s_k:=k^{-n/2}\sigma_{k,n}$.  We proved that $Z_k\Rightarrow Z$ and
$s_k\to\sigma_n^{\mathrm{sym}}>0$.  Rescaling
\eqref{eq:exact-small-ball}, with its arbitrary finite row shift chosen as
$k^{n/2}z$, gives, for every $r>0$ and all sufficiently large $k$,
\[
 \Pp\{\abs{Z_k-z}\le r\}
 \le \min\!\left\{1,\frac{B_{k,n}^{\R}}{s_k}\,r\right\}.
\]
The raw interval coefficient therefore converges as
$B_{k,n}^{\R}/s_k\to b_n^{\R}/\sigma_n^{\mathrm{sym}}$.

Fix $z\in\R$, $\varepsilon\ge0$, and $\delta>0$.  The interval centered at
$z$ with radius $(\varepsilon+\delta)\sigma_n^{\mathrm{sym}}$ is open.  The
Portmanteau theorem, the inclusion of the smaller closed interval, and the
preceding finite row bound give
\begin{align*}
 \Pp\{\abs{Z-z}\le\varepsilon\sigma_n^{\mathrm{sym}}\}
 &\le \Pp\{\abs{Z-z}<
       (\varepsilon+\delta)\sigma_n^{\mathrm{sym}}\}\\
 &\le \liminf_{k\to\infty}
       \Pp\{\abs{Z_k-z}<
       (\varepsilon+\delta)\sigma_n^{\mathrm{sym}}\}\\
 &\le \liminf_{k\to\infty}
       \min\!\left\{1,
       \frac{B_{k,n}^{\R}}{s_k}
       (\varepsilon+\delta)\sigma_n^{\mathrm{sym}}\right\}\\
 &=\min\{1,b_n^{\R}(\varepsilon+\delta)\}.
\end{align*}
The finite theorem applies for all sufficiently large $k$, which is enough
for the limit.  Letting $\delta\downarrow0$ proves the first shifted bound.
The stronger elementary envelope follows from \eqref{eq:bn-three-eighth}.
\end{proof}

\subsection{Density from the interval bound}
\label{app:symmetric-limit-density}
\begin{proof}[Proof of the density clauses of \cref{thm:symmetric-limit}]
Expose all edges not incident to vertex $2n$.  Define the columns
$C=(C_i)_{i=1}^{2n-1}$ and $g=(g_i)_{i=1}^{2n-1}$ in $\R^{2n-1}$ by
$C_i=\haf((S_n)_{\setminus\{i,2n\}})$ and $g_i=(S_n)_{i,2n}$.
Matching expansion gives $H_n=\langle g,C\rangle$.
The incident Gaussian edges are independent of $C$, so, with $V=\|C\|^2$,
$H_n\mid C\sim N(0,V)$.
The cofactor $C_{2n-1}$ equals one when its remaining vertices have just
the adjacent matching edges, each equal to one.  It is a nonzero polynomial
(the constant one for $n=1$), so the same zero-set argument as in
\cref{lem:positivity} gives $V>0$ almost surely.

The preceding interval argument, which did not use a density conclusion,
has already proved
\[
 \frac{\Pp\{|H_n|\le h\}}{2h}
 \le \frac{b_n^{\R}}{2\sigma_n^{\mathrm{sym}}},\qquad h>0.
\]
Apply \eqref{eq:mixture-fatou} with $R=\sqrt V$.  Substituting the
definition of $b_n^{\R}$ yields
\[
 \frac{\E V^{-1/2}}{\sqrt{2\pi}}
 \le\frac{b_n^{\R}}{2\sigma_n^{\mathrm{sym}}},\qquad
 \E V^{-1/2}\le\prod_{r=2}^n\gamma_{2r-1}<\infty.
\]
The direct density conclusions now follow from \cref{lem:mixture}, with
$c=\sigma_n^{\mathrm{sym}}$.  In particular,
$f_{\widehat H_n}(0)\le b_n^{\R}/2$.  This uses neither convergence of
densities nor an unproved independent-edge compression step.
\end{proof}

\section{Coefficient estimates and polynomial small balls}
\label{sec:constant}

The exact coefficient separates its order and row-dimension contributions.
For $k>n$, adjacent gamma ratios cancel to give
\begin{equation}
 \prod_{j=0}^{n-1}\gamma_{k-j}
 =2^{-n/2}\frac{\Gamma((k-n)/2)}{\Gamma(k/2)},\qquad
 B^{\R}_{k,n}=b_n^{\R}R_{k,n},\quad
 R_{k,n}:=\frac{\Gamma((k-n)/2)\sqrt{\Gamma(k/2+n)}}{\Gamma(k/2)^{3/2}}.
 \label{eq:gamma-cancellation}
\end{equation}
Indeed, the RMS row product is
$\prod_{q=0}^{n-1}(k+2q)=2^n\Gamma(k/2+n)/\Gamma(k/2)$ by the gamma
recurrence.  Multiplying its square root by the first expression proves
the factorization.

\subsection{A finite three-eighths bound}
We prove, for every $n\ge1$,
\begin{equation}
 \sqrt{2/\pi}\,n^{3/8}\le b_n^{\R}
 \le\frac2{\sqrt\pi}n^{3/8}.
 \label{eq:bn-three-eighth}
\end{equation}
Here is an explicit gamma-ratio argument.  For integers $j\ge1$, set
\[
 W_j=\sqrt{j-1/4}\,\frac{\Gamma(j)}{\Gamma(j+1/2)},\qquad
 Q_j=\frac{\Gamma(j+1)}{\sqrt{j+1/4}\,\Gamma(j+1/2)}.
\]
Both tend to one by Wendel's gamma-ratio limit \cite{Wendel1948}.
The recurrence gives
\[
 \frac{W_{j+1}^2}{W_j^2}
 =\frac{j^2(j+3/4)}{(j-1/4)(j+1/2)^2}>1,\qquad
 \frac{Q_{j+1}^2}{Q_j^2}
 =\frac{(j+1)^2(j+1/4)}{(j+5/4)(j+1/2)^2}<1.
\]
In the first fraction numerator minus denominator is $1/16$; in the
second it is $-1/16$.  Consequently $W_j\le1$ and $Q_j\ge1$.
The first inequality, substituted at $j=r-1$, gives
$\gamma_{2r-1}^2\le[2(r-5/4)]^{-1}$ for $r\ge2$.  Multiplication and
the gamma recurrence yield
\[
 (b_n^{\R})^2
 \le\frac4{\pi\sqrt\pi}
       \frac{\Gamma(n+1/2)\Gamma(3/4)}{\Gamma(n-1/4)}.
\]
Log convexity of $\Gamma$ gives
$\Gamma(x+3/4)\le x^{3/4}\Gamma(x)$ for $x>0$ by interpolation between
$x$ and $x+1$, and
$\Gamma(3/4)\le\sqrt{\Gamma(1/2)\Gamma(1)}=\pi^{1/4}\le\sqrt\pi$.
Taking $x=n-1/4$ proves the upper bound in \eqref{eq:bn-three-eighth}.
For the lower bound, $Q_j\ge1$ gives
\[
 \left(\frac{b_{j+1}^{\R}}{b_j^{\R}}\right)^2
 =(j+1/2)\frac{\Gamma(j)^2}{\Gamma(j+1/2)^2}
 \ge(1+1/(2j))(1+1/(4j))
 \ge(1+1/j)^{3/4}.
\]
The final inequality follows from concavity,
$(1+t)^{3/4}\le1+3t/4$ for $t\ge0$.  Telescope from
$b_1^{\R}=\sqrt{2/\pi}$ to finish the proof.

\subsection{The row-dimension envelope}
Cauchy--Schwarz and the chi-square density give, for $d>2$,
\begin{equation}
 \gamma_d^2\le\E(\chi_d^2)^{-1}=\frac1{d-2}.
 \label{eq:gamma-CS}
\end{equation}
For later comparison this also implies, by cancellation of odd integers,
\begin{equation}
 (2n-1)!!\prod_{r=2}^n\gamma_{2r-1}^2\le2n-1.
 \label{eq:odd-telescope}
\end{equation}
For $k\ge n+2$ and $0\le q<n$, \eqref{eq:gamma-CS} and $1+u\le e^u$
give
\begin{equation}
 (k+2q)\gamma_{k-q}^2
 \le 1+\frac{3q+2}{k-q-2}
 \le\exp\!\left\{\frac{3q+2}{k-n-1}\right\}.
 \label{eq:physical-bound}
\end{equation}
The product of the square roots is $R_{k,n}$, and
$\sum_{q=0}^{n-1}(3q+2)=(3n^2+n)/2$.  Hence
\[
 R_{k,n}\le\exp\!\left\{\frac{3n^2+n}{4(k-n-1)}\right\}.
\]
Combining this with \eqref{eq:bn-three-eighth} proves
\eqref{eq:elementary-B}.

\begin{proof}[Proof of \cref{cor:polynomial}]
The assumption gives $k_n/n\ge n/(D\log n)\to\infty$.  Thus the finite
theorem applies eventually and $k_n-n-1=k_n(1+o(1))$.  Taking logarithms
of \eqref{eq:elementary-B} gives
\[
 \log B^{\R}_{k_n,n}
 \le\log(2/\sqrt\pi)+\frac38\log n+
       \frac{3n^2+n}{4(k_n-n-1)}
 \le\left(\frac38+\frac{3D}{4}+o(1)\right)\log n.
\]
Multiplication by $n^{-\alpha}$ and the strict inequality in
\eqref{eq:alpha-choice} prove \eqref{eq:polynomial-small-ball}.
\end{proof}

\subsection{Exact comparison scales}
\label{sec:exact-comparison-scales}
The second coefficient also has an
exact equivalent.  Write $\mathcal G$ for the Barnes $G$ function,
satisfying $\mathcal G(z+1)=\Gamma(z)\mathcal G(z)$ and
$\mathcal G(1)=1$.  The definition of $\gamma_d$, the double factorial
identity, and the standard Barnes asymptotic give
\begin{equation}\label{eq:bnR-barnes-asymptotic}
 \frac{b_n^{\R}}2
 =\pi^{-3/4}\Gamma(n+1/2)^{1/2}
   \frac{\mathcal G(n)\mathcal G(3/2)}{\mathcal G(n+1/2)}
 \sim \frac{\mathcal G(1/2)}{\pi^{1/4}}n^{3/8}.
\end{equation}
Here $\mathcal G(3/2)=\sqrt\pi\,\mathcal G(1/2)$; the final equivalent follows
from the large argument expansion in \cite[Eq.~(5.17.5)]{NISTDLMFBarnesG}.

The two comparison scales satisfy, for every $n\ge1$,
\begin{equation}\label{eq:bnR-kappa-comparison}
 \frac{b_n^{\R}}2
 \le \sqrt{\frac{2n-1}{2\pi}}
 \le \kappa_{n,1}.
\end{equation}
The first inequality is the fixed degree form of the Gamma ratio estimate
in \eqref{eq:odd-telescope}.  For the second, put
$\rho_n:=\Gamma((n+1)/2)/\Gamma(n/2)$.  The duplication formula gives
$\kappa_{n,1}^2=n\rho_n/(2\sqrt\pi)$, while
$\rho_{n+2}=(n+1)\rho_n/n$.  The cases $n=1,2$ are direct, and the recurrence
together with $\rho_3=2/\sqrt\pi$ shows that
$\kappa_{n,1}^2\ge(2n-1)/(2\pi)$ for all $n\ge1$.

\section{A prescribed-pattern optical consequence}
\label{sec:real-optical}
Let $N=2n$, let $O_M$ be Haar orthogonal, and choose a deterministic or
independent output set $S\subseteq[M]$ of size $N$.  Suppose
$n+1\le K<M$, equal squeezing $r>0$ is applied to the first $K$ inputs,
and $NK/M\to0$.  With $\tau=\tanh r$ and
$X_M=\sqrt M\,O_{S,[K]}^{\T}\in\R^{K\times N}$, the optical kernel satisfies the exact
Gram identity
\[
 B_S=\frac{\tau}{M}X_M^{\T}X_M.
\]
Jahangiri, Arrazola, Quesada, and Killoran
\cite[Sec.~II~A]{JahangiriEtAl2020} explicitly treat real symmetric kernels
in Gaussian boson sampling.  For the real orthogonal interferometer used
here, their pure squeezed-state formula
\cite[Sec.~III~A~1, Eq.~(26)]{JahangiriEtAl2020} gives collision free
probabilities proportional to $\haf(B_S)^2$.
Writing $\delta_{M,N,K}$ for the total-variation distance from $X_M$ to
the iid standard real Gaussian $K\times N$ matrix, sparse-block
approximation gives $\delta_{M,N,K}\to0$
\cite[Theorem~1(i)]{JiangMa2019}\cite[Theorem~4]{Stewart2020}.
Homogeneity, total-variation contraction, and \cref{thm:main} imply
\[
 \sup_z\Pp\left\{|\haf(B_S)-z|\le
       \varepsilon(\tau/M)^n\sigma_{K,n}\right\}
 \le\min\{1,B^{\R}_{K,n}\varepsilon\}+\delta_{M,N,K}.
\]
Appendix~\ref{app:optical} gives the optical model, the full derivation,
and a conditional additive-to-relative consequence.  The approximation is
asymptotic; no inverse-polynomial rate for its error is assumed.

\section{Discussion and formal verification}
\label{sec:formalization}
The range $k\ge n+1$ suffices for a bounded density.  Boundedness in the
remaining range, strict comparison with the Pfaffian peak for higher
orders, and the sharp finite-dimension transition remain open.
The coefficient's $n^{3/8}$ order does not establish the same growth for
the actual hafnian density.  Unequal covariances and non-Gaussian columns
are further questions.

In the optical model, $K=M$ would give
$B=\tau I_M$ and zero nontrivial collision free hafnians, so the sparse
real sector cannot be extended to fully squeezed inputs by the same
Gaussian approximation.

The GitHub repository \cite{ZhaoRealGramHafniansGitHub2026} provides
unconditional Lean~4 verification of Theorems~\ref{thm:main}
and~\ref{thm:symmetric-limit}, using mathlib
\cite{deMouraUllrich2021,Mathlib2020} and no additional mathematical axioms.
The Zenodo release \cite{ZhaoLeanVerification2026} contains the broader
paper development, including supporting lemmas, equation endpoints,
classical benchmarks, optical transfer arguments, and the exact Barnes
identity and asymptotic \eqref{eq:bnR-barnes-asymptotic}.  Its correspondence
identifies every numbered theorem and equation and records whether it is
a definition, a hypothesis, a proved conclusion, or a cited input.
The full Pfaffian product law \eqref{eq:df-pfaffian-law-axiom} is the single
additional mathematical axiom, so its benchmark consequences are
conditionally verified.  The permanent bound and Haar-block approximation
are cited results outside the formal coverage; the optical transfer takes
an approximation bound as an explicit hypothesis.  Inventories also record
unnumbered and inline formulas, without asserting a separate Lean proof
for every expression.

\appendix

\section{The bilinear Gaussian integral}
\label{app:kernel}
\begin{proof}
Condition on $Y$ and integrate in $X$.  This gives
\begin{equation*}
 F_{t,T}(U,V)
 =\E_Y\exp\!\left\{-\frac{t^2}{2}\norm{TY+U}^2
 +itY\cdot V\right\}.
\end{equation*}
After expanding the norm and putting $E=I_d+t^2T^{\T}T$, the same expression is
\begin{equation*}
 e^{-t^2\norm U^2/2}
 \E_Y\exp\!\left\{-\frac12Y^{\T}(t^2T^{\T}T)Y
 +(-t^2T^{\T}U+itV)^{\T}Y\right\}.
\end{equation*}
For a real positive definite matrix $E\in\R^{d\times d}$ and a column
$w\in\mathbb C^d$, completing the square gives the Gaussian integral
\begin{equation*}
 (2\pi)^{-d/2}\int_{\R^d}
 e^{-y^{\T}Ey/2+w^{\T}y}\,\dd y
 =\det(E)^{-1/2}e^{w^{\T}E^{-1}w/2}.
\end{equation*}
For completeness, diagonalize $E$.  In each coordinate the identity follows
first for real $w$ by completing the square.  Both sides are entire functions
of that complex coordinate, so analytic continuation extends the identity to
complex $w$; multiplying the coordinate formulas gives the displayed
integral.
Apply it with $w=-t^2T^{\T}U+itV$.  Expanding
$w^{\T}E^{-1}w/2$ shows that the mixed term
$-it^3U^{\T}TE^{-1}V$ is purely imaginary.  Using
\begin{equation*}
 D^{-1}=I_d-t^2TE^{-1}T^{\T},
 \qquad \det D=\det E,
\end{equation*}
one obtains
\begin{align*}
 \abs{F_{t,T}(U,V)}
 &=\det(D)^{-1/2}
 \exp\!\left\{-\frac{t^2}{2}
 \big(U^{\T}D^{-1}U+V^{\T}E^{-1}V\big)\right\}.
\end{align*}
The endpoint formulas follow, and scaling the two quadratic forms by
$\theta$ and $1-\theta$ proves
\eqref{eq:geometric-interpolation}.
\end{proof}

\section{Classical Gaussian benchmarks}
\label{app:benchmarks}
\subsection{Cofactor conventions}
\label{app:four-cofactor}
For the permanent and determinant we use their usual unsigned and signed
permutation expansions.  The hafnian sums over perfect matchings.  The
Pfaffian uses the alternating matching signs in increasing vertex order;
equivalently it is the signed permutation sum divided by $2^n n!$
\cite[Section~0, Eqs.~(0.1)--(0.5)]{KnuthOverlappingPfaffians1996}.
All four polynomials of the empty matrix equal one.
Let $G_{i\mid n}$ delete row $i$ and column $n$, and let
$B_{\setminus\{i,2n\}}$ delete the indicated principal pair.  Define
\begin{equation}\label{eq:four-cofactor-vectors}
 \begin{aligned}
 C_i^P&:=\operatorname{per}(G_{i\mid n}),
       &&1\le i\le n,\\
 C_i^D&:=(-1)^{i+n}\det(G_{i\mid n}),
       &&1\le i\le n,\\
 C_i^H&:=\haf\bigl((S_n)_{\setminus\{i,2n\}}\bigr),
       &&1\le i\le2n-1,\\
 C_i^{\mathrm{Pf}}&:=(-1)^{i+1}
       \pf\bigl((A_n)_{\setminus\{i,2n\}}\bigr),
       &&1\le i\le2n-1.
 \end{aligned}
\end{equation}
Let $C^P,C^D\in\R^n$ and $C^H,C^{\mathrm{Pf}}\in\R^{2n-1}$ be the
corresponding column vectors.  Expose the Gaussian columns
$g^P,g^D\in\R^n$ and $g^H,g^{\mathrm{Pf}}\in\R^{2n-1}$ by
\begin{equation}\label{eq:four-exposed-gaussians}
 \begin{aligned}
 g^P=g^D&:=(G_{in})_{i=1}^{n},\\
 g^H&:=\bigl((S_n)_{i,2n}\bigr)_{i=1}^{2n-1},
 \qquad
 g^{\mathrm{Pf}}:=\bigl((A_n)_{i,2n}\bigr)_{i=1}^{2n-1}.
 \end{aligned}
\end{equation}
Thus the Ginibre cases expose the last column, and the symmetric and skew
symmetric cases expose the edges incident to the last vertex.  ``Unsigned''
for $C^P$ and $C^H$ means the absence of an alternating cofactor sign, not
that their random coordinates are nonnegative.

\begin{lemma}[The four cofactor expansions]
\label[lemma]{lem:four-cofactor}
For each $X=P_n,D_n,H_n,\mathrm{Pf}_n$, its displayed $g$ is standard
Gaussian and independent of $C$.  The scale $R_X=\|C\|$ is positive
almost surely, and
\begin{equation}
 X=\langle g,C\rangle,\qquad X\mid C\sim N(0,R_X^2).
 \label{eq:four-cofactor-conditional-law}
\end{equation}
Whenever $m_X=\E R_X^{-1}<\infty$, the conclusions of
\cref{lem:mixture} apply.  Explicitly, for $c>0$,
\begin{equation}
 f_{X/c}(x)=c f_X(cx)
 =\frac c{\sqrt{2\pi}}\E\left[R_X^{-1}e^{-c^2x^2/(2R_X^2)}\right],
 \label{eq:four-cofactor-density}
\end{equation}
\begin{equation}
 \|f_{X/c}\|_\infty=f_{X/c}(0)=\frac c{\sqrt{2\pi}}m_X,
 \label{eq:four-cofactor-peak}
\end{equation}
and
\begin{equation}
 \sup_z\Pp\{|X-z|\le c\varepsilon\}
 \le\min\{1,\sqrt{2/\pi}\,cm_X\varepsilon\}.
 \label{eq:four-cofactor-small-ball}
\end{equation}
\end{lemma}
\begin{proof}
Group permutations by the row assigned to column $n$, or matchings by
the partner of vertex $2n$.  The determinant cofactor sign is
$(-1)^{i+n}$; bringing the Pfaffian pair $(i,2n)$ to the front takes
$(i-1)+(2n-2)$ transpositions and gives $(-1)^{i+1}$.
This proves the four expansions.  Each cofactor uses only coordinates
outside the exposed column or incident edges, proving independence.
In the Ginibre cases the final cofactor equals one at an identity minor.
In the matching cases it equals one at the adjacent-edge matching.
It is therefore a nonzero polynomial, hence nonzero almost surely by
\cite{Mityagin2020ZeroSet}.  Conditioning and \cref{lem:mixture} finish
the proof.
\end{proof}

\subsection{The signed determinant law}
\label{ex:real-ginibre-determinant}
Gaussian Gram--Schmidt gives independent residual lengths with laws
$\chi_n,\ldots,\chi_1$: at each step the projected column is standard
Gaussian on the orthogonal complement, and its conditional radial law
does not depend on earlier columns.  Their product is the determinant
magnitude.  Negating one row preserves this magnitude and the matrix
law and reverses the sign.  The sign is thus fair and independent of
the nonzero magnitude.  A fair sign times $\chi_1$ is standard normal, so
\begin{equation}
 D_n\stackrel{\mathrm d}=Z\prod_{j=2}^n\chi_j,
 \qquad Z\sim N(0,1),\quad\text{all factors independent}.
 \label{eq:determinant-product-law}
\end{equation}
This is the Bartlett/QR product, also derived in
\cite[Proposition~2.1, Eq.~(10) and its proof]{KoehlerLeung2026}.
Their gamma variables at $\beta=1$ have shape $j/2$ and scale $2$,
equivalently chi squares of degree $j$.
Independence gives $\E D_n^2=n!$.  Applying \cref{lem:mixture} and
cancelling adjacent gamma ratios gives
\[
 f_{\widehat D_n}(0)
 =\frac{\sqrt{n!}}{\sqrt{2\pi}}\prod_{j=2}^n\gamma_j
 =\frac{\sqrt{n!}}{2^{n/2}\Gamma(n/2)}=\kappa_{n,1}.
\]
The same lemma gives continuity, the maximum at zero, and the shifted
interval coefficient $2\kappa_{n,1}$.

\subsection{The signed Pfaffian law}
\label{app:df-axiom}
Dumitriu and Forrester's Gaussian Householder reduction
\cite[Section~2, Eqs.~(2.3)--(2.4), and Eq.~(1.4) at $\beta=2$]{DumitriuForrester2010}
is the source for the following scalar consequence:
\begin{equation}
 \mathrm{Pf}_n\stackrel{\mathrm d}=Z\prod_{r=2}^n\chi_{2r-1},
 \qquad Z\sim N(0,1),\quad\text{all factors independent}.
 \label{eq:df-pfaffian-law-axiom}
\end{equation}
The source gives a matrix reduction; it does not state this signed scalar
law verbatim.  We spell out the scale and sign conversion.
Its $\widetilde\chi_j$ is the square root of a gamma variable of shape
$j/2$ and rate one, so $\sqrt2\widetilde\chi_j$ is our $\chi_j$.

The first exposed row has $2n-1$ independent unit-variance normal entries,
so its length is $\chi_{2n-1}$.  Conditional on that row, orthogonal
congruence preserves the independent trailing skew Gaussian block: its
density depends only on the Frobenius norm.  Its conditional law is
unchanged, proving independence from the extracted radius and allowing
iteration.  Thus the tridiagonal superdiagonal has independent radii
$\chi_{2n-1},\chi_{2n-2},\ldots,\chi_1$.
Only the adjacent matching contributes to a tridiagonal Pfaffian, so its
absolute value is the product of the odd radii.  Orthogonal congruence
preserves that absolute value.  Simultaneously negating the first row and
column reverses the Pfaffian and preserves the Gaussian law.  As in the
determinant argument, this supplies an independent fair sign, which combines
with the $\chi_1$ factor to give \eqref{eq:df-pfaffian-law-axiom}.
The formal supplement identifies this full signed law as the one external
mathematical axiom in the Pfaffian benchmark.

\begin{lemma}[Real Gaussian Pfaffian anticoncentration]
\label[lemma]{lem:pfaffian-benchmark}
The unit-variance Pfaffian satisfies
\begin{equation}
 \E\mathrm{Pf}_n^2=(2n-1)!!.
 \label{eq:pfaffian-second-moment}
\end{equation}
Its RMS-normalized density is continuous and even, maximal at zero,
with value $b_n^{\R}/2$.  For $\varepsilon\ge0$,
\begin{equation}
 \sup_z\Pp\{|\mathrm{Pf}_n-z|\le\varepsilon\sqrt{(2n-1)!!}\}
 \le\min\{1,b_n^{\R}\varepsilon\}.
 \label{eq:pfaffian-shifted-small-ball}
\end{equation}
\end{lemma}
\begin{proof}
In \eqref{eq:df-pfaffian-law-axiom}, the positive scale has inverse
moment $\prod_{r=2}^n\gamma_{2r-1}<\infty$, and its second moment is
$\prod_{r=2}^n(2r-1)$.  Independence gives the second-moment formula;
\cref{lem:mixture} gives all density and interval conclusions.  The
normalization identifies the peak as $b_n^{\R}/2$, which is at most
$n^{3/8}/\sqrt\pi$ by \eqref{eq:bn-three-eighth}.
\end{proof}

\section{The real optical model and probability transfer}
\label{app:optical}

The result above is a probability theorem for a real Gaussian Gram hafnian.
This section derives a separate prescribed pattern application to an equally
squeezed real optical sector.  The optical identity is exact; the replacement
of a sparse Haar orthogonal block by independent Gaussians is asymptotic.

\medskip
\noindent\textbf{Prescribed pattern optical consequence.}\quad
Let $N=2n$.  Along a sequence $M\to\infty$, choose
$K=K(M)<M$ with $K\ge n+1$ and $NK/M\to0$.  Let $O_M$ be Haar
distributed on $O(M)$, and let $S\subseteq[M]$, $\lvert S\rvert=N$, be
deterministic or chosen independently of $O_M$.  These are the assumptions
under which the probability estimate below is asserted.

The real field is not merely a technical simplification.  In a pure zero
displacement Gaussian boson sampler, prepare the first $K$ input modes in the
same squeezed vacuum $S(r)\lvert0\rangle$, with $r>0$ and common squeezing
phase zero, and prepare the remaining $M-K$ modes in vacuum.  Write
$[K]=\{1,\ldots,K\}$, put $\tau:=\tanh r$, and let $P_K:=I_K\oplus0_{M-K}$
be the projector onto the squeezed inputs.  Specializing the squeezed-state
construction of Jahangiri et al.\ \cite[Sec.~III~A~1, Eq.~(21)]{JahangiriEtAl2020}
to the real orthogonal interferometer $O_M$ gives the real symmetric kernel
\[
 B:=\tau O_M P_K O_M^{\T}.
\]
Define $B_S:=B[S,S]$, the principal submatrix of $B$ indexed by the prescribed
collision free output set $S$.  Since $B_S$ is real, the squared modulus in
the photon-counting formula \cite[Eq.~(26)]{JahangiriEtAl2020} is
$\haf(B_S)^2$.  General finite unitary
interferometers, and hence real orthogonal interferometers as a special case,
admit optical decompositions into transformations acting on two modes
\cite{ReckEtAl1994}.

Equal squeezing gives an exact Gram identity.  Set
\[
 X_M:=\sqrt M\,O_{S,[K]}^{\T}\in\R^{K\times N},
 \qquad
 B_S=\frac{\tau}{M}X_M^{\T}X_M,
 \qquad
 \haf(B_S)=\left(\frac{\tau}{M}\right)^n
             \haf(X_M^{\T}X_M).
\]
Here $O_{S,[K]}$ denotes the submatrix of $O_M$ whose rows are indexed by $S$
and whose columns are indexed by $[K]$.

\begin{figure}[!htbp]
  \centering
  \includegraphics[width=0.98\textwidth]{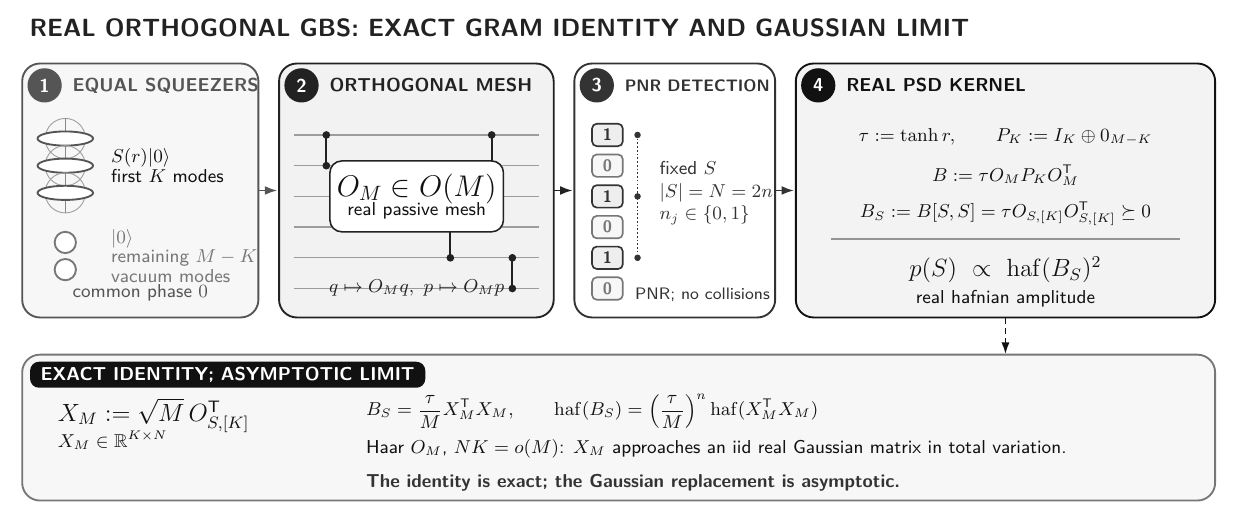}
  \caption{\textbf{Equal squeezing real orthogonal sector.}  The first $K$
  inputs are identically squeezed and the remaining inputs are vacuum; PNR
  denotes photon number resolving detection.  The $M$-component columns
  $q,p$ collect the position and momentum quadratures, respectively.
  The displayed Gram identity is
  exact, whereas the independent Gaussian replacement is asymptotic under
  $NK/M\to0$.}
  \label{fig:real-gbs-sector}
\end{figure}

Thus the common squeezing strength changes only the deterministic scale
$\tau^n$ of the hafnian.  The random matrix $X_M$ is not independent Gaussian
at finite $M$.  By invariance of Haar measure under row permutations,
$O_{S,[K]}$ has the same law as the $N\times K$ upper left block.  Therefore,
if $O_M$ is Haar orthogonal, Jiang and Ma
\cite[Theorem~1(i)]{JiangMa2019} and, independently, Stewart
\cite[Theorem~4]{Stewart2020} give
\[
 \delta_{M,N,K}:=
 d_{\mathrm{TV}}\!\left(\mathcal L(X_M),\mathcal L(X)\right)\longrightarrow0
 \quad\text{along }M\to\infty\text{ with }NK/M\to0,
\]
where $X\in\R^{K\times N}$ has independent standard normal entries
and we use the convention
$d_{\mathrm{TV}}(\mu,\nu)=\sup_A|\mu(A)-\nu(A)|$.  Here their ambient
dimension and block dimensions are respectively $M,N,K$; transposition
preserves total variation.  Only convergence to zero is used, so choosing the
other common normalization of total variation would not change the
conclusion.  The hypothesis $NK/M\to0$ is automatic for fixed
$N,K$ and makes the block dependence asymptotically negligible.  Total
variation cannot increase under the map $Y\mapsto\haf(Y^{\T}Y)$.
Consequently, when
$K\ge n+1$, the main theorem
implies the asymptotic optical bound
\[
 \sup_{z\in\R}\mathbb P_{O_M}\!\left\{
 \left|\haf(B_S)-z\right|
 \le \varepsilon\left(\frac{\tau}{M}\right)^n\sigma_{K,n}\right\}
 \le \min\{1,B^{\R}_{K,n}\varepsilon\}+\delta_{M,N,K}.
\]

For this prescribed output pattern, the same estimate gives a conditional
additive to relative transfer for the collision free probability.  The pure
squeezed-state formula \cite[Eq.~(26)]{JahangiriEtAl2020}, with its
normalization $(\operatorname{sech}r)^K$ for the $K$ squeezed inputs,
and the exact Gram identity give
\[
 p_S:=(\operatorname{sech}r)^K\haf(B_S)^2,
 \qquad
 p_{\mathrm{ref}}:=(\operatorname{sech}r)^K
 \left(\frac{\tau}{M}\right)^{2n}\sigma_{K,n}^2,
\]
where $p_{\mathrm{ref}}$ is the surrogate mean scale, namely the exact mean
in the independent Gaussian model.  Hence, for $t\ge0$,
\[
 \{p_S\le t p_{\mathrm{ref}}\}
 =\left\{\left|\haf(X_M^{\T}X_M)\right|
 \le \sqrt t\,\sigma_{K,n}\right\},
 \qquad
 \mathbb P_{O_M}\{p_S\le t p_{\mathrm{ref}}\}
 \le \min\{1,B^{\R}_{K,n}\sqrt t+\delta_{M,N,K}\}.
\]
More generally, let $\eta\ge0$ and $0\le\gamma\le1$, and suppose that an
estimator $\widetilde p_S$, possibly randomized and dependent on the optical
instance, satisfies
\[
 \mathbb P\{|\widetilde p_S-p_S|>\eta p_{\mathrm{ref}}\}\le\gamma.
\]
For every $\rho>0$, we have the deterministic event inclusion
\[
 \{|\widetilde p_S-p_S|>\rho p_S\}
 \subseteq
 \{|\widetilde p_S-p_S|>\eta p_{\mathrm{ref}}\}
 \cup\{p_S\le(\eta/\rho)p_{\mathrm{ref}}\}.
\]
Combining this inclusion with the preceding estimate gives
\[
 \mathbb P\{|\widetilde p_S-p_S|>\rho p_S\}
 \le \min\!\left\{1,\gamma+B^{\R}_{K,n}\sqrt{\eta/\rho}
       +\delta_{M,N,K}\right\}.
\]
Probability is joint over the Haar orthogonal matrix and the estimator's
internal randomness.  No independence assumption on the estimator is needed.
The square root is
the real probability penalty: the real amplitude has a small ball bound
linear in its radius, whereas the output probability is its square.  This
calculation uses only the zero centered specialization of the stronger shifted
theorem.  It converts an available additive guarantee into a relative one; it
does not itself provide an additive estimator, an average case hardness
theorem, or a noise robust quantum advantage reduction.

The approximation error is additive and only known here to be $o(1)$.
The prescribed-pattern hypothesis excludes adaptive selection based on
the interferometer.  The main text explains the degeneracy at $K=M$;
the present conclusion requires the stated sparse-block regime.
\FloatBarrier

\section*{Code availability}
The unconditional Lean~4 proofs of Theorems~\ref{thm:main}
and~\ref{thm:symmetric-limit} are available in the
\href{https://github.com/HongruZhao/RealGramHafnians}{GitHub repository}
\cite{ZhaoRealGramHafniansGitHub2026}.
The broader development, conditional benchmark proofs, and theorem/equation
correspondence are archived on Zenodo \cite{ZhaoLeanVerification2026}, with
the coverage described in Section~\ref{sec:formalization}.  Both releases
provide pinned sources and build instructions.  The Zenodo identifier is
\href{https://doi.org/10.5281/zenodo.22498111}{10.5281/zenodo.22498111}.
The Zenodo source release is licensed under GPL-3.0-only.

\section*{Competing interests}
The author declares no competing financial or nonfinancial interests.

\section*{AI-assisted preparation}
OpenAI Codex (GPT-5.6 Sol and GPT-6 Astra Ultra) assisted under the author's
direction with literature organization, LaTeX restructuring, language
editing, proof exploration, and the implementation and debugging of
the paper-facing Lean endpoints identified in the companion correspondence.  Codex also assisted in running the
pinned Lean build and transitive axiom audits; the resulting proof terms were
checked by the Lean kernel.

\section*{Funding}
No specific funding was received for this work.

\renewcommand{\bibfont}{\small}
\begingroup
\interlinepenalty=10000
\bibliographystyle{plainnat}
\bibliography{bibliography}
\endgroup

\end{document}